\documentclass[preprint,times]{elsarticle}
\usepackage{enumerate}
\usepackage{amsmath,amsthm,amssymb,amsfonts}
\usepackage{mathtools}
\usepackage{graphicx}
\usepackage{color}
\usepackage{subfig}
\usepackage{url}
\usepackage{hyperref}
\usepackage{tikz}
\usetikzlibrary{positioning}
\usetikzlibrary{shapes,arrows}
\usepackage{algorithmic}
\usepackage{algorithm}

\usepackage{booktabs}
\usepackage{multirow}
\usepackage[T1]{fontenc}  % change font
\usepackage{newpxtext, newpxmath}  % change font

\usepackage{comment}
\usepackage{geometry}
\usepackage{pifont} % ×記号用

\newcommand{\mini}{\mathop{\rm min.}\limits}
\newcommand{\maxi}{\mathop{\rm max.}\limits}

\newcommand{\st}{{\rm s.~t.~}}

\newcommand{\vt}[1]{\boldsymbol{#1}}
\newcommand{\bb}[1]{\mathbb{#1}}

\newcommand{\xmark}{\ding{55}}   % バツマーク

\theoremstyle{plain}
	  \newtheorem{thm}{Theorem}
	  
	  \newtheorem{lemma}{Lemma}
	  
\theoremstyle{definition}
	  
	  \newtheorem{ass}{Assumption}
	  
	  \newtheorem{axiom}{Axiom}
\theoremstyle{remark}

\allowdisplaybreaks

\journal{Elsevier}

\begin{document}
\begin{frontmatter}

  % -------- Title, Authors, Abstract, Keywords -------------------------------- %
%\section*{[Dummy] Title, Authors, Abstract, Keywords}
\title{%
    An Axiomatic Analysis of Distributionally Robust Optimization with $q$-Norm Ambiguity Sets for Probability Smoothing
    %\tnotemark[1]
}

%\tnotetext[1]{Minor revision. Assumption~\ref{ass:don-degen} in Lemma~\ref{lem:lambda0} has been removed, as the result holds without it.}

\author[econ]{Yoichi Izunaga\corref{cor}}
\ead{izunaga@econ.kyushu-u.ac.jp}

\author[jgmi]{Kota Kurihara}
\ead{kurihara.kota.881@s.kyushu-u.ac.jp}

\author[econ_grad]{Hokuto Nagano}
\ead{nagano.hokuto.397@s.kyushu-u.ac.jp}

\author[econ_grad]{Daiki Uchida}
\ead{uchida.daiki.098@s.kyushu-u.ac.jp}

\address[econ]{%
    Faculty of Economics,
    Kyushu University,
    Fukuoka,
    Japan
}

\address[jgmi]{%
    Joint Graduate School of Mathematics for Innovation,
    Kyushu University,
    Fukuoka,
    Japan
}

\address[econ_grad]{%
    Graduate School of Economics,
    Kyushu University,
    Fukuoka,
    Japan
}

\cortext[cor]{Corresponding author.}

%----- abstract -------------------------------
\begin{abstract}
We analyze the axiomatic properties of a class of probability estimators derived from Distributionally Robust Optimization (DRO) with $q$-norm ambiguity sets ($q$-DRO), a principled approach to the zero-frequency problem.
While classical estimators such as Laplace smoothing are characterized by strong linearity axioms like Ratio Preservation, we show that $q$-DRO provides a flexible alternative that satisfies other desirable properties.
We first prove that for any $q \in [1, \infty]$, the $q$-DRO estimator satisfies the fundamental axioms of Positivity and Symmetry.
For the case of $q \in (1, \infty)$, we then prove that it also satisfies Order Preservation.
Our analysis of the optimality conditions also reveals that the $q$-DRO formulation is equivalent to the regularized empirical loss minimization.
\end{abstract}

\end{frontmatter}

%----Introduction----------------------
\section{Introduction}\label{sec:intro}

The estimation of probabilities from finite data is a fundamental task of machine learning, statistics, and information theory.
A common and persistent challenge in this task is the zero-frequency problem: if an event is not observed in a finite sample,
its probability is naively estimated as zero, leading to poor generalization and model failure (e.g.,~\citet{CG_CSL1999, WB_IEEE2002}). 
This issue is critical in diverse fields, from natural language processing, where unseen $N$-grams cause serious problems for language models~\citep{MS_FSNLP1999}, to risk management, where the possibility of unobserved catastrophic events must be accounted for.

The classical remedy for this problem is Laplace smoothing (or Add-one smoothing), a simple technique that adds a pseudocount to every category.
While effective, its justification was long considered heuristic.
Recently, this perspective has been challenged by an axiomatic characterization proving that Laplace smoothing is the unique method satisfying a set of four intuitive axioms: Positivity, Symmetry, Order Preservation, and Ratio Preservation~\citep{S_MSS2025}.
However, this characterization also highlighted a crucial limitation.
The Ratio Preservation axiom imposes a strong linear structure on the estimator, which can be overly rigid for complex, real-world data.

This rigidity has a clear interpretation within a Bayesian framework.
It is well-established that Laplace smoothing is mathematically equivalent to a Bayesian posterior mean when assuming a uniform prior distribution over the space of all possible probability distributions~\citep{MS_FSNLP1999}.
This prior embodies the simple belief that all probability distributions are a priori equally likely.
The rigidity of Laplace smoothing is, therefore, a direct consequence of the simplicity of its underlying prior belief.

The rigidity of the classical approach raises a critical question:
Can we design a more flexible and principled smoothing method that is not bound by such a rigid prior,
yet still satisfies the most desirable axiomatic properties?

This paper provides an affirmative answer by leveraging the framework of Distributionally Robust Optimization (DRO).
Instead of specifying an explicit prior,
DRO formulates estimation as a min-max game against an adversary who selects the worst-case probability distribution from an ambiguity set centered around the empirical distribution~\citep{BN_RO2009, KSW_DRO2025}.
We specifically analyze a DRO model where the ambiguity set is defined by the $q$-norm (hereafter, $q$-DRO).

Our contributions are as follows:
\begin{itemize}
  \item We formulate the $q$-DRO smoothing problem and show that it can be reformulated as a single convex conic optimization problem.
  \item We provide an axiomatic analysis of the $q$-DRO estimator.
  We prove that it satisfies the fundamental axioms of Positivity and Symmetry for all $q \in [1,\infty]$. Our main axiomatic result is a proof that for $q \in (1, \infty)$, the estimator also satisfies Order Preservation, under a mild assumption reflecting a non-trivial problem setting.
  \item We show that the $q$-DRO formulation can be interpreted as a form of regularized empirical loss minimization. While the equivalence between DRO and regularized empirical loss minimization is well established in the literature~\citep{SMPK_NIPS2015, DN_AS2021}, our contribution lies in identifying the specific regularization structure induced by a $q$-norm ambiguity set on the probability simplex.
\end{itemize}

Our work bridges three distinct fields, robust optimization, axiomatic analysis, and regularized empirical loss minimization, to present DRO as a principled framework for designing estimators that are robust, axiomatically sound, and theoretically justified.

The remainder of this paper is organized as follows.
Section~\ref{sec:prelim} reviews the existing axiomatic approach to probability smoothing and formally introduces our $q$-DRO framework.
Section~\ref{sec:reformulation} demonstrates that the proposed $q$-DRO problem can be reformulated as a tractable convex conic optimization problem.
Section~\ref{sec:main_results} presents our main theoretical results, establishing that the $q$-DRO estimator satisfies the fundamental axioms.
Section~\ref{sec:discussion} discusses theoretical implications, including the validity of our assumptions and the connection to regularized empirical loss minimization.
Section~\ref{sec:numerical_examples} presents numerical examples to validate our theoretical findings and illustrates the behavior of the estimator.
Finally, Section~\ref{sec:conclusion} concludes the paper.

%---Preliminaries---%
\section{Preliminaries: From Axiomatic Smoothing to a Distributionally Robust Formulation}\label{sec:prelim}

%---Axiomatic approach
\subsection{The Axiomatic Approach to Probability Smoothing}
Let $N=\{1,2,\ldots,n\}$ be a set of categories.
A probability distribution is a vector $\vt{p}$ in the probability simplex $\Delta^{n}=\{ \vt{p} \in \bb{R}^{n} \mid \sum_{j=1}^{n}p_{j}=1, p_{j} \ge 0\, (j \in N) \}$.
A smoothing function is a map $f: D \to \text{ri}(\Delta^{n})$ such that
\begin{align*}
  f(\vt{p})=(f_{1}(\vt{p}), f_{2}(\vt{p}), \ldots, f_{n}(\vt{p})) \quad \text{for any } \vt{p} \in D,
\end{align*}
where $D \subset \Delta^{n}$ is the domain of empirical distributions,
typically those with at least one zero component\footnote{Regarding the domain $D$, \citet{S_MSS2025} states that ``it refers to any non-empty subset of $\Delta^{n}$, without any additional assumptions imposed.'' However, if the uniform distribution is included in $D$, the characterization does not hold.},
and $\text{ri}(\Delta^{n})$ is the relative interior of the simplex, ensuring that $f_{j}(\vt{p})>0$ for all $j \in N$.
Laplace smoothing is a smoothing function such that for any empirical distribution $\hat{\vt{p}} \in D$,
\begin{align*}
  f_{j}(\hat{\vt{p}})=\frac{\hat{p}_{j}+c}{1+n\cdot c} \quad (j \in N),  
\end{align*}
where $c > 0$ is a user-specified parameter as the pseudocount.

We consider the following axioms from~\citep{S_MSS2025}.
\begin{axiom}{(Positivity)}
  For any $\hat{\vt{p}} \in D$ and $i \in N$, $f_{i}(\hat{\vt{p}})>0$.
\end{axiom}

\begin{axiom}{(Symmetry)}
  For any $\hat{\vt{p}} \in D$ and $i,j \in N$, if $\hat{p}_{i}=\hat{p}_{j}$, then $f_{i}(\hat{\vt{p}})=f_{j}(\hat{\vt{p}})$.
\end{axiom}

\begin{axiom}{(Order Preservation)}
  For any $\hat{\vt{p}} \in D$ and $i,j \in N$, if $\hat{p}_{i}<\hat{p}_{j}$, then $f_{i}(\hat{\vt{p}})<f_{j}(\hat{\vt{p}})$.
\end{axiom}

\begin{axiom}{(Ratio Preservation)}
  For any $\hat{\vt{p}} \in D$ and $i,j,k \in N$, if $\hat{p}_{i} \neq \hat{p}_{j}, \hat{p}_{k}$, then 
  \begin{align*}
    \frac{f_{i}(\hat{\vt{p}})-f_{j}(\hat{\vt{p}})}{\hat{p}_{i}-\hat{p}_{j}}
    &=\frac{f_{i}(\hat{\vt{p}})-f_{k}(\hat{\vt{p}})}{\hat{p}_{i}-\hat{p}_{k}}.
  \end{align*}
\end{axiom}

It is obvious that ratio preservation implies symmetry (by Lemma~1 in~\citep{S_MSS2025}),
and~\citep{S_MSS2025} showed that the only method satisfying positivity, order preservation, and ratio preservation is Laplace smoothing.
The ratio preservation axiom, however, imposes a strong linear structure on the estimator, which may not be suitable for some applications.

%---DRO framework
\subsection{A Principled Alternative: Distributionally Robust Optimization}

We consider an alternative approach based on DRO, which does not rely on pre-specified behavioral axioms. DRO approaches the smoothing problem by directly modeling uncertainty in the empirical distribution $\hat{p} \in \Delta^n$~\citep{KSW_DRO2025}.
It frames the problem as a game between a player and an adversary.
The player chooses an estimator $\vt{x}$, while the adversary chooses the "true" distribution $\vt{p}$ from an ambiguity set $\mathcal{U}$ of distributions close to $\hat{\vt{p}}$, with the goal of maximizing the player's loss $\mathcal{L}(\vt{x},\vt{p})$.
Hence, the player's problem is formulated as finding the estimator that minimizes this worst-case loss:
\begin{align*}
  \underset{\vt{x} \in \Delta^{n}}{\mini} \underset{\vt{p} \in \mathcal{U}(\hat{\vt{p}})}{\maxi} \mathcal{L}(\vt{x},\vt{p}).
\end{align*}

While this min-max formulation is more formally described as a two-person zero-sum game, we refer to it as DRO throughout this paper. 
This aligns our work with the common paradigm in machine learning where DRO is often interpreted as a form of regularized empirical loss minimization, a connection we will make explicit in Section~\ref{sec:discussion}.

%---Our Formulation
\subsection{Our DRO Formulation for Probability Smoothing}
In this paper, we specify the components of the DRO game as follows. The player's loss $\mathcal{L}$ is measured by the cross-entropy loss
\begin{align*}
  \mathcal{L}(\vt{x},\vt{p})=\sum_{j=1}^{n}p_{j}(-\log{x_{j}}).
\end{align*}
This choice is well-motivated by its connection to the principle of Maximum Likelihood Estimation.
From an information-theoretic perspective, it corresponds to minimizing the Kullback-Leibler (KL) divergence from the empirical distribution $\hat{\vt{p}}$ to the model distribution $\vt{x} \in \Delta^{n}$~\citep{CT_EIT2006}.
It is also motivated by the work in natural language processing (e.g., \citet{BPP_CL1999}), which successfully used information-theoretic objectives for probability estimation.
Intuitively, the logarithmic term penalizes any assignment of zero probability with an infinite loss, thus structurally enforcing the goal of smoothing: to avoid zero-probability estimates.

The ambiguity set is defined by perturbations to the empirical distribution.
We introduce a perturbation $e_{j}$ for each category $j \in N$ to represent the deviation from the empirical probability, such that the probability is defined as $p_{j}=\hat{p}_{j}+e_{j}$ for $j \in N$.
The ambiguity set is then formed by all such distributions $\vt{p}=(p_{j})_{j \in N}$ where the magnitude of the perturbation is bounded by a $q$-norm $\|\vt{e}\|_{q}=\left( \sum_{j=1}^{n}|e_{j}|^{q} \right)^{1/q}$ for any $q \in [1,\infty]$:
\begin{align}  
  \mathcal{U}_{q}(\hat{\vt{p}}, \varepsilon)=\left\{\, \vt{p} \in \Delta^{n} \mid p_{j}=\hat{p}_{j}+e_{j}\,(j \in N), \|\vt{e}\|_{q} \le \varepsilon \,\right\}. \label{eq:ambiguity1}
\end{align}
Here, $\varepsilon > 0$ is a user-specified parameter, known as the robustness radius, that controls the size of the ambiguity set.
A larger $\varepsilon$ implies a higher degree of uncertainty about the empirical distribution, leading to a more robust and conservative estimator.
The problem is formulated as a min-max game, denoted by $q$-DRO,
\begin{align*}
  \underset{\vt{x} \in \Delta^{n}}{\mini} \underset{\vt{p} \in \mathcal{U}(\hat{\vt{p}},\varepsilon)}{\maxi} \left\{ \sum_{j=1}^{n}p_{j}(-\log{x_{j}}) \right\}.
\end{align*}

Since any distribution $\vt{p}$ must belong to the probability simplex $\Delta^{n}$, two conditions must hold: (i) $p_{j} \ge 0$ for all $j \in N$, and (ii) $\sum_{j=1}^{n}p_{j}=1$.
The first condition directly implies that the perturbations must satisfy $\hat{p}_{j}+e_{j} \ge 0$.
The second condition implies that the sum of the perturbations must be zero, as shown by a simple calculation:
\begin{align*}
  \sum_{j=1}^{n}p_{j}=\sum_{j=1}^{n}(\hat{p}_{j}+e_{j})=1+\sum_{j=1}^{n}e_{j}.
\end{align*}
For this to equal $1$, we must have $\sum_{j=1}^{n}e_{j}=0$.
Thus, the perturbation vector $\vt{e}$ is required to satisfy the following three constraints:
\begin{align}
  \hat{p}_{j}+e_{j} \ge 0 \; (j \in N),\quad \sum_{j=1}^{n}e_{j}=0,\quad \|\vt{e}\|_{q} \le \varepsilon. \label{eq:ambiguity2}
\end{align}
These constraints provide an explicit characterization of the ambiguity set in terms of the perturbation vector $\vt{e}$.

%---Reformulation of q-DRO---%
\section{Reformulation of $q$-DRO}\label{sec:reformulation}
Using the explicit characterization of the ambiguity set~\eqref{eq:ambiguity2}, the inner worst-case problem of the $q$-DRO formulation, for a fixed estimator $\vt{x} \in \Delta^{n}$, can be stated as:
\begin{subequations}\label{prob:wc}
  \begin{align}
    \underset{\vt{e}}{\maxi} \quad & \sum_{j=1}^{n}(\hat{p}_{j}+e_{j})(-\log{x_{j}}) \label{eq:wc_obj}\\
    \st \quad & \hat{p}_{j}+e_{j} \ge 0 \quad (j \in N), \label{eq:wc_con1}\\
    & \sum_{j=1}^{n}e_{j}=0, \label{eq:wc_con2}\\
    & \|\vt{e}\|_{q} \le \varepsilon. \label{eq:wc_con3}
  \end{align}
\end{subequations}
This is a convex optimization problem that satisfies Slater's condition, which guarantees that strong duality holds.
To demonstrate this, we only need to show that a strictly feasible point exists.
If all empirical probabilities are positive ($\hat{p}_{j}>0$ for all $j \in N$), then  the zero vector $\vt{e}=\vt{0}$ is a strictly feasible point.
If some categories have zero frequency (i.e., $\hat{p}_{j}=0$ for some $j \in N$), a strictly feasible point can be constructed by a small perturbation.
For instance, one can assign a sufficiently small positive probability to the zero-frequency categories, and subtract a corresponding amount from the positive-frequency categories such that the sum of perturbations remains zero. For a sufficiently small perturbation, all inequality constraints will be satisfied strictly.

Therefore, we leverage Lagrange duality to convert the inner worst-case problem into an equivalent minimization problem.
To derive the dual of the inner worst-case problem~\eqref{prob:wc}, we introduce a vector of nonnegative Lagrangian multiplier $\lambda_{j} \ge 0$ for the first constraint~\eqref{eq:wc_con1} ($\hat{p}_{j}+e_{j} \ge 0$) and a multiplier $\beta \in \bb{R}$ for the second constraint~\eqref{eq:wc_con2} ($\sum_{j=1}^{n}e_{j}=0$).
The Lagrangian $L(\vt{e},\vt{\lambda}, \beta)$ for the inner worst-case problem, explicitly retaining the norm constraint~\eqref{eq:wc_con3}, is as follows:
\begin{align*}
  L(\vt{e},\vt{\lambda}, \beta)
  &=\sum_{j=1}^{n}(\hat{p}_{j}+e_{j})(-\log{x_{j}})+\sum_{j=1}^{n}\lambda_{j}(\hat{p}_{j}+e_{j})-\beta(\sum_{j=1}^{n}e_{j})\\
  &=\sum_{j=1}^{n}(-\log{x}_{j}+\lambda_{j}-\beta)e_{j}+\sum_{j=1}^{n}\hat{p}_{j}(-\log{x_{j}}+\lambda_{j}).
\end{align*}
The Lagrange dual function $g(\vt{\lambda},\beta)$ is derived by maximizing $L(\vt{e},\vt{\lambda},\beta)$ with respect to $\vt{e}$ over the remaining constraint, $\|\vt{e}\|_{q} \le \varepsilon$:
\begin{align*} % one-column version
  g(\vt{\lambda},\beta) = \max & \left\{ \sum_{j=1}^{n}(-\log{x}_{j}+\lambda_{j}-\beta)e_{j} \mid \| \vt{e} \|_{q} \le \varepsilon \right\}+\sum_{j=1}^{n}\hat{p}_{j}(-\log{x_{j}}+\lambda_{j}).
\end{align*}
% \begin{align*} % two-column version
%   g(\vt{\lambda},\beta) = \max & \left\{ \sum_{j=1}^{n}(-\log{x}_{j}+\lambda_{j}-\beta)e_{j} \mid \| \vt{e} \|_{q} \le \varepsilon \right\}\\
%   &+\sum_{j=1}^{n}\hat{p}_{j}(-\log{x_{j}}+\lambda_{j}).
% \end{align*}
The above maximization problem can be solved analytically
in terms of the dual norm $\|\cdot\|_{q}^{\bullet}$ corresponding to $q$-norm (see, e.g.,~\citet{B_CO2004}), where the dual norm is defined as:
\begin{align}
  \|\vt{y}\|_{q}^{\bullet}=\max \{ \vt{y}^{\top}\vt{x} \mid \|\vt{x}\|_{q} \le 1 \}. \label{eq:dual_norm}
\end{align}
The term $\max \{ \vt{c}^{\top}\vt{e} \mid \|\vt{e}\|_{q} \le \varepsilon\}$ evaluates to $\varepsilon \cdot \|\vt{c}\|_{q}^{\bullet}$ by~\eqref{eq:dual_norm}.
This yields the closed-form expression for the dual function:
\begin{align*}
  g(\vt{\lambda},\beta) &= \sum_{j=1}^{n}\hat{p}_{j}(-\log{x_{j}}+\lambda_{j}) + \varepsilon \cdot \| -\log(\vt{x})-\beta\vt{1}+\vt{\lambda} \|_{q}^{\bullet},
\end{align*}
where $\vt{1}$ is a vector of all ones, $\log(\vt{x})$ denotes component-wise application of the logarithm, $\log(\vt{x}) = (\log x_1, \log x_2, \ldots, \log x_n)^\top$, and $\vt{\lambda}=(\lambda_1, \lambda_2, \ldots, \lambda_n)^\top$.

Since strong duality holds, the optimal value of the inner worst-case problem is equal to the minimum of the dual function $g(\vt{\lambda},\beta)$ over the dual variables $(\vt{\lambda},\beta)$.
Consequently, the original min--max problem reduces to a single minimization problem, yielding the following reformulation of $q$-DRO:
\begin{subequations}\label{prob:q-DRO} % one-column version
  \begin{align}
    \underset{(\vt{x},\vt{\lambda},\beta)}{\mini} \quad & \sum_{j=1}^{n}\hat{p}_{j}(-\log{x_{j}}+\lambda_{j}) + \varepsilon \cdot \| -\log(\vt{x})-\beta\vt{1}+\vt{\lambda} \|_{q}^{\bullet} \label{eq:q-DRO_obj}\\
    \st \quad & \sum_{j=1}^{n}x_{j}=1, \label{eq:q-DRO_con1}\\
    & \lambda_{j} \ge 0 \quad (j \in N), \label{eq:q-DRO_con2}\\
    & x_{j} \ge 0 \quad (j \in N). \label{eq:q-DRO_con3}
  \end{align}
\end{subequations}
%  \begin{subequations}\label{prob:q-DRO} % two-column version
%   \begin{align}
%     \underset{(\vt{x},\vt{\lambda},\beta)}{\mini} \quad & \sum_{j=1}^{n}\hat{p}_{j}(-\log{x_{j}}+\lambda_{j}) \notag\\
%     &+ \varepsilon \cdot \| -\log(\vt{x})-\beta\vt{1}+\vt{\lambda} \|_{q}^{\bullet} \label{eq:q-DRO_obj}\\
%     \st \quad & \sum_{j=1}^{n}x_{j}=1, \label{eq:q-DRO_con1}\\
%     & \lambda_{j} \ge 0 \quad (j \in N), \label{eq:q-DRO_con2}\\
%     & x_{j} \ge 0 \quad (j \in N). \label{eq:q-DRO_con3}
%   \end{align}
% \end{subequations}

The dual norm $\|\cdot\|_{q}^{\bullet}$ is given as follows:
\begin{align*}
  \|\vt{y}\|_{q}^{\bullet}&=
  \begin{cases}
    \|\vt{y}\|_{\infty} = \max \left\{ |y_{j}| \mid j=1,2,\ldots,n \right\} & (q=1)\\
    \|\vt{y}\|_{q^{*}}= \left( \sum_{j=1}^{n}|y_{j}|^{q^{*}} \right)^{1/q^{*}} & (q \in (1,\infty))\\
    \|\vt{y}\|_{1} = \sum_{j=1}^{n}|y_{j}| & (q=\infty)
  \end{cases},
\end{align*}
where $q^{*}$ is the dual exponent satisfying $1/q + 1/q^{*}=1$.

Although the $q^{*}$-norm function $g(\vt{u})=\|\vt{u}\|_{q^{*}}$ is convex for any $q^{*} \in [1,\infty]$, the composite function $g(\vt{u}(\vt{x},\beta, \vt{\lambda}))$ is not necessarily convex when $\vt{u}$ is defined as $\vt{u}(\vt{x},\beta, \vt{\lambda})=-\log(\vt{x})-\beta\vt{1}+\vt{\lambda}$.
Nevertheless, $q$-DRO can be formulated as a standard convex conic optimization problem by introducing auxiliary variables and conic constraints.
See \ref{app:reformulation} for details.
This reformulation enables efficient computation of a globally optimal solution using off-the-shelf solvers.

Specifically, the logarithmic terms $-\log x_j$ in the cross-entropy loss can be represented via exponential cone constraints.
Moreover, the $q^{*}$-norm term admits standard conic representations depending on $q$:
it reduces to linear constraints for $q\in\{1,\infty\}$,
and power cone constraints for general $q\in(1,\infty)$.
In particular, for $q=2$, it reduces to second-order cone constraints.
See~\citet{Mosek_cookbook} for more details.

%---Main Results---%
\section{Main Results: Axiomatic Properties of the $q$-DRO Estimator}\label{sec:main_results}

In this section, we analyze structural properties of optimal solutions of $q$-DRO.
From now on, we refer to an optimal solution $\vt{x}$ of $q$-DRO~\eqref{prob:q-DRO} as the $q$-DRO estimator.
We first establish Positivity for all $q\in[1,\infty]$.

\begin{thm}\label{thm:Pos}
  For any $q \in [1,\infty]$, any $q$-DRO estimator $\vt{x}$ satisfies Positivity, i.e., $x_j>0$ for all $j\in N$.
\end{thm}

\begin{proof}
Assume for contradiction that there exists an optimal solution $(\vt{x}, \beta, \vt{\lambda})$ such that $x_j = 0$ for some $j \in N$.
Note that a feasible solution with a finite objective value exists (e.g., $(\bar{\vt{x}}, \bar{\beta},\bar{\vt{\lambda}})=(\vt{1}/n, 0,\vt{0})$), so the optimal objective value implies finiteness.

First, if $\hat{p}_j > 0$, the term $\hat{p}_j (-\log x_j)$ diverges to $+\infty$, which contradicts the finiteness of the optimal value.
Thus, we assume $\hat{p}_j = 0$ and let $\vt{u} = -\log(\vt{x}) - \beta\vt{1} + \vt{\lambda}$. Since $x_j = 0$, the term $-\log x_j$ is $+\infty$.
For the norm $\|\vt{u}\|_{q^*}$ to remain finite, the component $u_j = -\log x_j - \beta + \lambda_j$ must be finite.
Since $\lambda_j \ge 0$, this requires $\beta$ to tend to $+\infty$ to counteract the divergence of $-\log x_j$.
However, for any index $k$ with $\hat{p}_k > 0$, we have $x_k > 0$ and $\lambda_k < +\infty$ to ensure the finiteness of the term $\hat{p}_k (-\log x_k + \lambda_k)$ in the objective.
Consider such an index $k$.
If $\beta \to +\infty$, then the component $u_k = -\log x_k - \beta + \lambda_k$ diverges to $-\infty$, since $-\log x_k$ and $\lambda_k$ are finite values.
Consequently, $|u_k| \to \infty$, causing $\|\vt{u}\|_{q^*} \to \infty$.
This contradicts the finiteness of the optimal value.
Therefore, there cannot exist an optimal solution with $x_j = 0$. Hence, $x_j > 0$ for all $j \in N$.
\end{proof}

We now analyze Symmetry and Order Preservation.
The behavior of the optimal solution $(\vt{x},\beta,\vt{\lambda})$ depends on whether the norm term $\| -\log(\vt{x})-\beta\vt{1}+\vt{\lambda} \|_{q^{*}}$ equals zero or not.

\subsection{Degenerate Case: $\|-\log(\vt{x})-\beta\vt{1}+\vt{\lambda}\|_{q^{*}}=0$}\label{sec:degenerate}

Suppose $\|-\log(\vt{x})-\beta\vt{1}+\vt{\lambda}\|_{q^{*}}=0$ at an optimal solution.
Equivalently, $-\log x_j-\beta+\lambda_j=0$ for all $j\in N$.
Then~\eqref{prob:q-DRO} simplifies to
  \begin{subequations} % one-column version
  \begin{align*}
    \underset{(\vt{x},\beta)}{\mini} \quad & \beta \\
    \st \quad & \sum_{j=1}^{n}x_{j}=1,\\
    & \lambda_{j}=\log{x}_{j}+\beta \ge 0 \quad (j \in N),\\
    & x_{j} \ge 0 \quad (j \in N).
  \end{align*}
  \end{subequations}
  This problem admits the explicit optimal solution $(x_{1},\ldots,x_{n},\beta)=(1/n, \ldots, 1/n, \log{n})$.
  Therefore, the estimator is the uniform distribution and Symmetry holds trivially, whereas Order Preservation does not hold in general.

\subsection{Non-degenerate case: $\|-\log(\vt{x})-\beta\vt{1}+\vt{\lambda}\|_{q^{*}}>0$}

Next, we discuss an optimal solution $(\vt{x}, \beta, \vt{\lambda})$ which satisfies the following non-degeneracy assumption.
\begin{ass}\label{ass:don-degen}
  At an optimal solution $(\vt{x},\beta,\vt{\lambda})$, we assume
$\|-\log(\vt{x})-\beta\vt{1}+\vt{\lambda}\|_{q^{*}}>0$.
\end{ass}

To investigate Symmetry and Order Preservation, we introduce the Lagrangian for $q$-DRO.
Let $\xi_{j} \ge 0$ be the Lagrangian multiplier for the constraint~\eqref{eq:q-DRO_con2} ($\lambda_{j} \ge 0$ for each $j \in N$), and $\gamma \in \bb{R}$ be the multiplier for the constraint~\eqref{eq:q-DRO_con1} ($\sum_{j=1}^{n}x_{j}=1$).
Note that the positivity of $x_{j}$ is guaranteed by Theorem~\ref{thm:Pos}, so we do not need to introduce multipliers for the constraints~\eqref{eq:q-DRO_con3} ($x_{j} \ge 0$) due to the complementarity condition.
The Lagrangian $L(\vt{x},\vt{\lambda},\beta,\vt{\xi},\gamma)$ for $q$-DRO is given by
\begin{align*} % one-column version
  L(\vt{x},\vt{\lambda},\beta,\gamma,\vt{\xi})
  &=\sum_{j=1}^{n}\hat{p}_{j}(-\log{x_{j}}+\lambda_{j})+\varepsilon \cdot \| -\log(\vt{x})-\beta\vt{1}+\vt{\lambda} \|_{q^{*}}+\gamma ( \sum_{j=1}^{n}x_{j}-1 )-\sum_{j=1}^{n}\xi_{j}\lambda_{j}.
\end{align*}
% \begin{align*} % two-column version
%   L(\vt{x},\vt{\lambda},\beta,\gamma,\vt{\xi})
%   &=\sum_{j=1}^{n}\hat{p}_{j}(-\log{x_{j}}+\lambda_{j})\\
%   &+\varepsilon \cdot \| -\log(\vt{x})-\beta\vt{1}+\vt{\lambda} \|_{q^{*}}\\
%   &+\gamma ( \sum_{j=1}^{n}x_{j}-1 )-\sum_{j=1}^{n}\xi_{j}\lambda_{j}.
% \end{align*}

For $q^{*}\in(1,\infty)$, the $q^{*}$-norm is differentiable except at $-\log(\vt{x})-\beta\vt{1}+\vt{\lambda}=\vt{0}$, which is excluded by Assumption~\ref{ass:don-degen}.
For $q\in\{1,\infty\}$, it is nonsmooth at certain points, so we adopt generalized KKT conditions with subgradients.

Since the inner function $(\vt{x},\beta,\vt{\lambda})\mapsto \vt{u}$ is continuously differentiable whenever $\vt{x}$ is positive and the outer function $\vt{u}\mapsto\|\vt{u}\|_{q^{*}}$ is convex,
the subdifferential chain rule is applicable (see, e.g., \citet{RW_VA1998}).
In particular, when $q\in(1,\infty)$ the subdifferential is a singleton (the gradient).

The generalized KKT conditions state that at an optimal solution $(\vt{x},\vt{\lambda},\beta)$, there exist multipliers $(\vt{\xi},\gamma)$ and a subgradient $\vt{y}^{(q^{*})}\in\partial\|\vt{u}\|_{q^{*}}$ such that the following conditions hold:
\begin{subequations}\label{eq:kkt}
\begin{align}
  -\frac{\hat{p}_j}{x_j} +\gamma - \frac{\varepsilon y_j^{(q^{*})}}{x_j}&=0 \quad (j \in N), \label{eq:kkt1}\\
  \hat{p}_{j}-\xi_{j}+\varepsilon y_j^{(q^{*})} &=0 \quad (j \in N), \label{eq:kkt2}\\
  \varepsilon \cdot \sum_{j=1}^{n} y_{j}^{(q^{*})}&=0,\label{eq:kkt3}\\
  \xi_{j}\cdot \lambda_{j} &=0 \quad (j \in N).\label{eq:kkt4}
\end{align}
\end{subequations}

For brevity, we omit the primal and dual feasibility conditions.
Note that the above KKT conditions are well-defined since Theorem~\ref{thm:Pos} ensures $x_j>0$.

We next give explicit subgradient forms used in our analysis.

- For $q^{*}=1$ (i.e., $q=\infty$), the subgradient $ \vt{y}^{(1)}\in\partial\|\vt{u}\|_{1}$ is given by:
\[
y_j^{(1)}\in
\begin{cases}
\{\mathrm{sgn}(u_j)\}, & u_j\neq0,\\
[-1,1], & u_j=0.
\end{cases}
\]

- For $q^{*}\in(1,\infty)$ (i.e., $q\in(1,\infty)$), the $q^{*}$-norm is differentiable at any $\vt{u} \neq \vt{0}$. Thus, under Assumption~\ref{ass:don-degen}, the subgradient $\vt{y}^{(q^{*})}$ is uniquely determined as the gradient:
\[
y_j^{(q^{*})}=\dfrac{|u_j|^{q^{*}-1}\mathrm{sgn}(u_j)}{\|\vt{u}\|_{q^{*}}^{q^{*}-1}}.
\]

- For $q^{*}=\infty$ (i.e., $q=1$), let $I(\vt{u})=\{j\in N:\ |u_j|=\|\vt{u}\|_{\infty}\}$.
Under Assumption~\ref{ass:don-degen}, the maximum absolute value $\|\vt{u}\|_{\infty}$ is positive, which implies $|u_j| > 0$ for all $j \in I(\vt{u})$. Consequently, $\mathrm{sgn}(u_j)$ is uniquely determined as either $1$ or $-1$ for any index $j \in I(\vt{u})$.
Then $\vt{y}^{(\infty)}\in\partial\|\vt{u}\|_{\infty}$ is given as
\[
y_j^{(\infty)}=
\begin{cases}
c_j\,\mathrm{sgn}(u_j), & j\in I(\vt{u}),\\
0, & j\notin I(\vt{u}),
\end{cases}
\]
where there exists coefficient $c_j\ge 0$ such that $\sum_{j\in I(\vt{u})}c_j=1$.

We now show the following key lemma.
\begin{lemma}\label{lem:lambda0}
  Let $q \in [1, \infty]$.
  In any optimal solution $(\vt{x},\vt{\lambda},\beta)$ to $q$-DRO,
  $\lambda_{j}=0$ for all categories $j \in N$. 
\end{lemma}
\begin{proof}
  For an optimal solution which does not satisfy Assumption~\ref{ass:don-degen}, the degenerate analysis in Section~\ref{sec:degenerate} yields the uniform solution $(x_1,\dots,x_n)=(1/n,\dots,1/n)$ and $\beta=\log n$.
  Hence, $\lambda_j=\log x_j+\beta=\log(1/n)+\log n=0$ for all $j \in N$.

  Assume now that Assumption~\ref{ass:don-degen} holds at an optimal solution.
  The stationarity condition~\eqref{eq:kkt1} of the KKT conditions with respect to $x_{j}$ yields
  \begin{align}
    -\hat{p}_{j}+\gamma x_{j}=\varepsilon y_{j}^{(q^{*})} \label{eq:kkt1-2}
  \end{align}
  for any $j \in N$.
  Summing up the above equation~\eqref{eq:kkt1-2} and together with \eqref{eq:kkt3}, we have
  \begin{align*}
    0=\varepsilon \sum_{j=1}^{n} y_{j}^{(q^{*})}=-\sum_{j=1}^{n}\hat{p}_{j}+\gamma \sum_{j=1}^{n}x_{j}=-1+\gamma,
  \end{align*}
  which implies $\gamma=1$.
  The sum of \eqref{eq:kkt2} and \eqref{eq:kkt1-2}, we have $\xi_{j}=x_{j}>0$ for any $j \in N$ owing to the positivity of $x_{j}$.
  Therefore, $\lambda_{j}$ equals to zero from the complementarity condition~\eqref{eq:kkt4}.
\end{proof}

From Lemma~\ref{lem:lambda0}, the subgradient $\vt{y}^{(q^{*})} \in \partial \|\vt{u}\|_{q^{*}}$ satisfies a monotonicity property with respect to the component $u_j=-\log{x_j}-\beta$.
Specifically, if $u_i<u_j$, then $y_i^{(q^{*})} \le y_j^{(q^{*})}$ holds for any $q \in [1,\infty]$.
Moreover, for $q \in (1,\infty)$, if $u_i \le u_j$, then $y_i^{(q^{*})} \le y_j^{(q^{*})}$.
The detailed proof of this monotonicity property is provided in~\ref{app:monotonicity}. Based on this monotonicity, we derive the following results.

\begin{thm}\label{thm:Sym}
  Let $q \in [1, \infty]$.
  Any $q$-DRO estimator $\vt{x}$ satisfies Symmetry.
\end{thm}

\begin{proof}
  First, consider an optimal solution that does not satisfy Assumption 1.
  As discussed in Section~\ref{sec:degenerate}, the optimal solution $\vt{x}$ corresponds to the uniform distribution, which trivially satisfies Symmetry.

  Next, we consider an optimal solution that satisfies Assumption 1. Suppose that $\vt{x}$ satisfies $x_{i}>x_{j}$ while $\hat{p}_{i}=\hat{p}_{j}$ for some $i,j \in N$.
  Taking the difference between the $i$-th and $j$-th equations of \eqref{eq:kkt1-2}, we obtain
  \begin{align}
  \varepsilon(y_i^{(q^{*})} - y_j^{(q^{*})})=(\hat{p}_j - \hat{p}_i) + (x_i-x_j)=x_i-x_j>0. \label{thm2}
  \end{align}

  On the other hand, from Lemma~\ref{lem:lambda0} and the assumption $x_i > x_j$, we have $-\log{x_i}-\beta=u_i < u_j=-\log{x_{j}}-\beta$.
  Due to the monotonicity of the subgradient, the inequality $u_i < u_j$ implies $y_i^{(q^*)} \le y_j^{(q^*)}$.
  This contradicts~\eqref{thm2}.
\end{proof}

\begin{thm}\label{thm:OP}
  Let $q\in(1,\infty)$. Under Assumption~\ref{ass:don-degen}, any $q$-DRO estimator $\vt{x}$ satisfies Order Preservation.
\end{thm}

\begin{proof}
Assume for contradiction that $\hat{p}_i<\hat{p}_j$ but $x_i\ge x_j$ for some $i,j\in N$.
Similar to the proof of Symmetry, taking the difference of~\eqref{eq:kkt1-2}, we obtain
\begin{align}
  \varepsilon\bigl(y_i^{(q^{*})}-y_j^{(q^{*})}\bigr)
  =(\hat{p}_j-\hat{p}_i)+(x_i-x_j) > 0. \label{thm3}
\end{align}
From the assumption $x_i \ge x_j$ and Lemma~\ref{lem:lambda0}, we have $-\log{x_i}-\beta=u_i \le u_j=-\log{x_{j}}-\beta$.

Here, under Assumption~\ref{ass:don-degen}, the $q$-norm is differentiable for $q \in (1, \infty)$, and the subgradient coincides with the gradient. Due to the monotonicity of the gradient, the inequality $u_i \le u_j$ implies $y_i^{(q^*)} \le y_j^{(q^*)}$, which contradicts~\eqref{thm3}.
\end{proof}

\textbf{Remark}. The above argument does not directly extend to the boundary cases, $q=1$ and $q=\infty$.
Nevertheless, a weaker form of Order Preservation can be established: if $\hat{p}_i<\hat{p}_j$, then $x_i\le x_j$ holds for any $q\in[1,\infty]$.
Furthermore, we construct counterexamples where Order Preservation fails for these boundary cases.
The details of these counterexamples are provided in Section~\ref{sec:discussion}.

We summarize the results in Table~\ref{tab:summary}.
\begin{table}[htb]
  \centering
  \caption{Axiomatic properties of $q$-DRO estimator.
  A check mark (\checkmark) indicates that the property is satisfied, and a cross (\xmark) indicates that it is not.}
  \begin{tabular}{cccccc}
    \toprule
    & Degenerate Case & & \multicolumn{2}{c}{Non-Degenerate Case} \\
    \cline{2-2} \cline{4-5}
    & $q \in [1, \infty]$ & & $q \in \{1,\infty\}$ & $q \in (1, \infty)$ \\
    \midrule
    %$\vt{\lambda}=\vt{0}$ & \checkmark & & \checkmark & \checkmark \\
    Positivity & \checkmark & & \checkmark & \checkmark \\
    Symmetry & \checkmark & & \checkmark & \checkmark \\
    Order Preservation & \xmark & & \xmark & \checkmark \\
    weak Order Preservation & \checkmark & & \checkmark & \checkmark \\
    \bottomrule
  \end{tabular}
  \label{tab:summary}
\end{table}

%---Discussion---%
\section{Discussion}\label{sec:discussion}
Our axiomatic analysis reveals that $q$-DRO estimators form a flexible class of smoothing rules.
The analysis in Section~\ref{sec:main_results} further clarifies the theoretical foundations, specifically addressing the validity of the non-degeneracy assumption and establishing the equivalence of the problem~\eqref{prob:q-DRO} to regularized empirical loss minimization.

\subsection{Validity of Assumption}

Assumption~\ref{ass:don-degen} (i.e., $\| -\log(\vt{x})-\beta\vt{1}+\vt{\lambda} \|_{q^{*}}>0$) was introduced as a technical condition to ensure the gradient of the $q^{*}$-norm is well-defined for $q \in (1, \infty)$.
We now show that this assumption is not merely technical, but reflects a natural property of the problem, by analyzing the KKT conditions of the inner worst-case problem~\eqref{prob:wc} from Section~\ref{sec:reformulation}.

Let $\nu$ be a nonnegative Lagrangian multiplier for the norm constraint~\eqref{eq:wc_con3} ($\|\vt{e}\|_{q} \le \varepsilon$).
The Lagrangian for the problem~\eqref{prob:wc} is given as
\begin{align*} % one-column version
  L(\vt{e},\vt{\lambda}, \beta, \nu)
  &=\sum_{j=1}^{n}(-\log{x}_{j}-\beta+\lambda_{j})e_{j}+\sum_{j=1}^{n}\hat{p}_{j}(-\log{x_{j}}+\lambda_{j})+\nu(\varepsilon-\|\vt{e}\|_{q}),
\end{align*}
% \begin{align*} % two-column version
%   L(\vt{e},\vt{\lambda}, \beta, \nu)
%   &=\sum_{j=1}^{n}(-\log{x}_{j}-\beta+\lambda_{j})e_{j}\\
%   &+\sum_{j=1}^{n}\hat{p}_{j}(-\log{x_{j}}+\lambda_{j})+\nu(\varepsilon-\|\vt{e}\|_{q}),
% \end{align*}
then the stationarity condition with respect to $e_{j}$ and the complementarity condition imply
\begin{subequations}\label{eq:wc_kkt}
  \begin{align}
  \partial_{e_{j}}L=(-\log{x}_{j}-\beta+\lambda_{j})-\nu d_j^{(q)}&=0 \quad (j \in N),\label{eq:wc_kkt1}\\
  \nu \cdot (\varepsilon-\|\vt{e}\|_{q})&=0 \label{eq:wc_kkt2},
\end{align}
\end{subequations}
where $\vt{d}^{(q)} \in \partial\|\vt{e}\|_{q}$ is a subgradient of the $q$-norm.

Suppose $\| -\log(\vt{x})-\beta\vt{1}+\vt{\lambda} \|_{q^{*}}>0$. If $\nu$ were zero, then equation~\eqref{eq:wc_kkt1} would imply $-\log x_{j}-\beta+\lambda_j=0$ for all $j \in N$, which contradicts the assumption that the norm is positive. Therefore, $\nu$ must be positive.

On the other hand, if $\nu>0$, then $\|\vt{e}\|_{q}=\varepsilon$ from~\eqref{eq:wc_kkt2}, which implies $\vt{e}\neq \vt{0}$.
Since $\vt{e}$ is nonzero, its subgradient $\vt{d}^{(q)} \in \partial\|\vt{e}\|_{q}$ must be nonzero as well.
Specifically, there exists at least one category $j \in N$ such that $d^{(q)}_j \neq 0$.
Hence, $\|-\log(\vt{x})-\beta\vt{1}+\vt{\lambda} \|_{q^{*}}>0$ from~\eqref{eq:wc_kkt1}.

This analysis reveals that Assumption 1 is equivalent to the condition $\nu > 0$.
By the complementarity condition~\eqref{eq:wc_kkt2},
$\nu > 0$ implies that the norm constraint is active,
i.e., $\|\vt{e}\|_q = \varepsilon$. This means the adversary perturbs the distribution to the maximum allowed radius $\varepsilon$.

This confirms that the assumption is not merely technical, but reflects a natural characteristic of the problem setting: it holds whenever the robustness radius $\varepsilon$ is not set to a value so large that the solution is forced to be uniform. In other words, the assumption remains valid as long as the player trusts the empirical distribution $\hat{\vt{p}}$ to some extent as an anchor for smoothing.

\subsection{Equivalence to Regularized Empirical Loss Minimization}

From Lemma~\ref{lem:lambda0}, we obtain that the dual variables $\lambda_{j}$ are zero for any $j \in N$.
This implies that the objective function of $q$-DRO simplifies:
\begin{align*}
  \sum_{j=1}^{n}\hat{p}_{j}(-\log{x_{j}})+\varepsilon \cdot \| -\log(\vt{x})-\beta\vt{1} \|_{q^{*}},
\end{align*}
which is a form of regularized empirical loss minimization.
The term $\sum_{j \in N} \hat{p}_j(-\log x_j)$ corresponds to the empirical cross-entropy loss, while the norm term $\varepsilon \cdot \| -\log(\vt{x})-\beta\vt{1} \|_{q^{*}}$ acts as a regularizer.
This regularizer has a clear interpretation.
The variable $\beta$ acts as a baseline (reference) value for the cost ($-\log x_j$) of all categories.
The regularization term then penalizes the deviation of each category's cost from this baseline value $\beta$.

\subsection{Analysis of Boundary Cases: $q=1$ and $q=\infty$}
We numerically investigated the Order Preservation for the boundary cases where the proof in Theorem~\ref{thm:OP} does not directly apply due to the lack of strict monotonicity in the subgradients.

\paragraph{Case~1: $q=1$ ($q^{*}=\infty$)}
We examined the instance with $n=4$, $\hat{\vt{p}} = (0.00, 0.07, 0.465, 0.465)^{\top}$, and $\varepsilon=0.3$.
The conic optimization solver MOSEK returned a solution $\vt{x} = (0.11, 0.11, 0.39, 0.39)^{\top}$.
In this solution, we observe $x_1 = x_2$ despite the strict inequality in the empirical distribution ($\hat{p}_1 < \hat{p}_2$).
We verified that this solution $\vt{x}$ satisfies the optimality conditions of the reformulated convex problem of $1$-DRO.
Thus, this is a definitive counterexample to Order Preservation.

\paragraph{Case~2: $q=\infty$ ($q^{*}=1$)}
Similarly, the case for $q=\infty$ provides a counterexample.
We considered $n=4$, $\hat{\vt{p}} = (0.0, 0.2, 0.3, 0.5)^{\top}$, and $\varepsilon=0.2$. MOSEK returned the solution $\vt{x} = (0.20, 0.25, 0.25, 0.30)^{\top}$.
Here, we observe $x_2 = x_3 = 0.25$ even though $\hat{p}_2 < \hat{p}_3$. We also verified that this solution $\vt{x}$ satisfies the optimality conditions of the reformulated convex problem of $\infty$-DRO.
This confirms that strict Order Preservation does not hold for $q=\infty$.

This result is a direct consequence of the sparsity-inducing property of the dual $1$-norm regularizer.
Minimizing the $1$-norm encourages multiple components of the vector $-\log \vt{x} - \beta \vt{1}$ to become exactly zero simultaneously (i.e., $-\log x_j = \beta$).
This implies $x_2 = x_3 = e^{-\beta}$.
Thus, the $1$-norm actively suppresses the differences in empirical frequencies, assigning identical probabilities to distinct categories.

%---Numerical Examples--%
\section{Numerical Experiments}\label{sec:numerical_examples}

This section presents numerical experiments to validate our theoretical findings.
Specifically, we conduct the following experiments:
(i) numerically confirm that the $q$-DRO estimator satisfies the axioms for various values of $q$,
(ii) examine the effect of the robustness radius $\varepsilon$ on the optimal solution.
All experiments are implemented in Python using MOSEK as the conic optimization solver~\citep{Mosek_11}.

\subsection{Experiment~1: Validation of Axiomatic Properties}

We first numerically confirm that the $q$-DRO estimator for $q \in \{1.5, 2, 3\}$ satisfies the axioms of Positivity, Symmetry, and Order Preservation.
We set $n=5$ and the robustness radius $\varepsilon=0.2$.
The empirical distribution is set as $\hat{\vt{p}} = (0.00, 0.15, 0.15, 0.30, 0.40)^{\top}$ to test all axioms.
This distribution includes a zero-frequency category ($\hat{p}_1=0$), categories with identical frequencies ($\hat{p}_2 = \hat{p}_3$), and categories with distinct frequencies ($\hat{p}_1 < \hat{p}_2 < \hat{p}_4 < \hat{p}_5$).

Let $\vt{x}^{(q)}$ denote the optimal solution of $q$-DRO.
The optimal solutions for $q \in \{1.5, 2, 3\}$ are as follows:
\begin{align*}
\vt{x}^{(1.5)} &= (0.1216, 0.1643, 0.1643, 0.2548, 0.2950)^{\top},\\
\vt{x}^{(2)} &= (0.1342, 0.1792, 0.1792, 0.2332, 0.2742)^{\top},\\
\vt{x}^{(3)} &= (0.1542, 0.1927, 0.1927, 0.2123, 0.2481)^{\top}.
\end{align*}

These results confirm our theoretical findings:
(i) the zero-frequency category $\hat{p}_1=0.00$ is assigned a positive probability $x_1^{(q)}>0$,
(ii) the categories with identical empirical frequencies, $\hat{p}_2 = \hat{p}_3 = 0.15$, receive equal probabilities, $x_2^{(q)} = x_3^{(q)}$,
and (iii) the input order $\hat{p}_1 < \hat{p}_2  < \hat{p}_4 < \hat{p}_5$ is strictly preserved in the output: $x_1^{(q)} < x_2^{(q)} < x_4^{(q)} < x_5^{(q)}$ for all $q \in \{1.5, 2, 3\}$.

\subsection{Experiment~2: Sensitivity Analysis}

Next, we analyze the effect of the robustness radius $\varepsilon$ (regularization strength).
We use $n=4$ categories and a simple empirical distribution $
\hat{\vt{p}} = (0.10, 0.20, 0.30, 0.40)^{\top}.$
We fix $q=2$ and vary $\varepsilon$ from 0.0 to 0.3.
\begin{figure}[tbhp]
  \centering
  \includegraphics[scale=0.5]{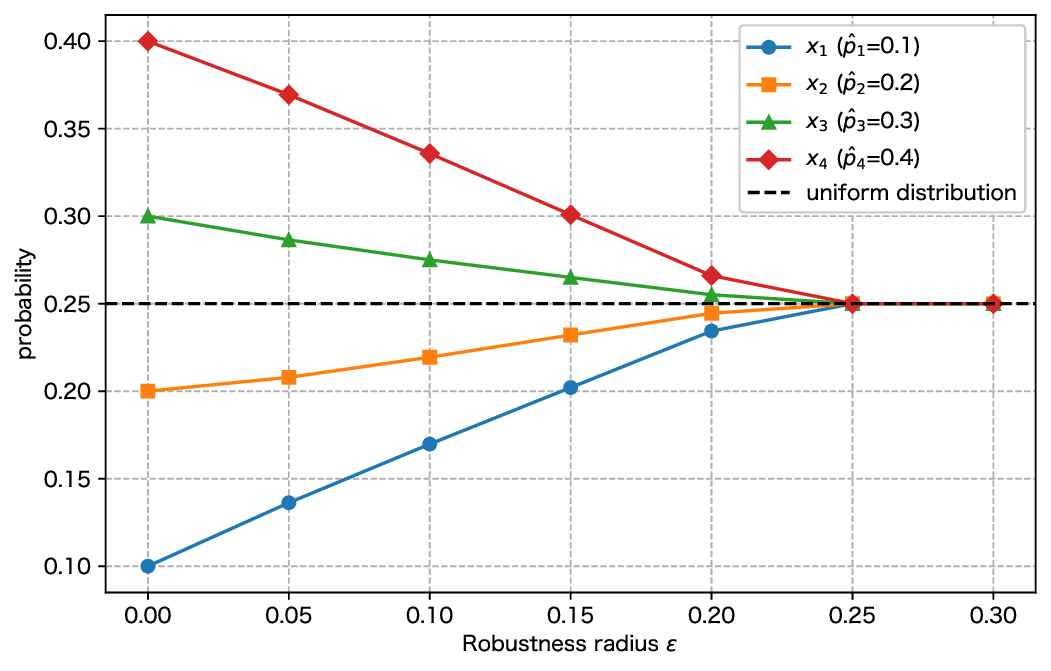}
  \caption{Sensitivity of $2$-DRO estimator $\vt{x}$ to the robustness radius $\varepsilon$.}
  \label{fig:exp3a}
\end{figure}

Figure~\ref{fig:exp3a} illustrates how each component $x_j$ changes as the robustness radius $\varepsilon$ varies.
For every category, the corresponding curve shows the trajectory of the estimated probability $x_j$ as the regularization strength increases.

At $\varepsilon=0$, the solution is identical to the empirical distribution, $\vt{x} = \hat{\vt{p}}$.
As $\varepsilon$ increases, the probabilities shrink toward the uniform distribution ($p_j=0.25$).
This confirms that $\varepsilon$ controls the trade-off between fitting to the data $\hat{\vt{p}}$ and robustness (regularization towards uniformity).

%---Conclusion---%
\section{Conclusion and Future Work}\label{sec:conclusion}

This paper analyzed the axiomatic properties of probability estimators derived from distributionally robust optimization with 
$q$-norm ambiguity sets.
We established that the resulting $q$-DRO estimator satisfies Positivity and Symmetry for all $q \in [1,\infty]$, and further proved that Order Preservation holds for all $q \in (1,\infty)$ under a mild non-degeneracy assumption.
Our analysis of the KKT conditions clarified how the structure of the dual variables leads to a clear interpretation of the $q$-DRO formulation as a form of regularized empirical loss minimization.

Directions for future work are as follows.
First, investigating the geometric and statistical meaning of the regularization term, and its Bayesian interpretation, would be a valuable extension.
Second, it would be valuable to investigate the behavior of DRO estimators under other types of ambiguity sets, such as those defined by the Wasserstein distance~\citep{MK_MM2018}.
Comparing the axiomatic properties of these variants with $q$-DRO could provide a more comprehensive understanding of robust smoothing techniques.

\section*{Acknowledgement}
This work is supported by JSPS Grant-in-Aid (22K17856).

\bibliographystyle{abbrvnat}
\bibliography{ref}

\begin{thebibliography}{15}
\providecommand{\natexlab}[1]{#1}
\providecommand{\url}[1]{\texttt{#1}}
\expandafter\ifx\csname urlstyle\endcsname\relax
  \providecommand{\doi}[1]{doi: #1}\else
  \providecommand{\doi}{doi: \begingroup \urlstyle{rm}\Url}\fi

\bibitem[Ben-Tal et~al.(2009)Ben-Tal, Ghaoui, and Nemirovski]{BN_RO2009}
A.~Ben-Tal, L.~E. Ghaoui, and A.~Nemirovski.
\newblock \emph{Robust Optimization}.
\newblock Princeton Series in Applied Mathematics. Princeton University Press, 2009.

\bibitem[Berger et~al.(1996)Berger, Della~Pietra, and Della~Pietra]{BPP_CL1999}
A.~Berger, S.~A. Della~Pietra, and V.~J. Della~Pietra.
\newblock A maximum entropy approach to natural language processing.
\newblock \emph{Computational Linguistics}, 22\penalty0 (1):\penalty0 39--71, 1996.

\bibitem[Boyd and Vandenberghe(2004)]{B_CO2004}
S.~Boyd and L.~Vandenberghe.
\newblock \emph{Convex Optimization}.
\newblock Cambridge University Press, 2004.

\bibitem[Chen and Goodman(1999)]{CG_CSL1999}
S.~F. Chen and J.~Goodman.
\newblock An empirical study of smoothing techniques for language modeling.
\newblock \emph{Computer Speech \& Language}, 13\penalty0 (4):\penalty0 359--394, 1999.

\bibitem[Cover and Thomas(2006)]{CT_EIT2006}
T.~M. Cover and J.~A. Thomas.
\newblock \emph{Elements of Information Theory (Wiley Series in Telecommunications and Signal Processing)}.
\newblock Wiley-Interscience, USA, 2006.

\bibitem[Duchi and Namkoong(2021)]{DN_AS2021}
J.~C. Duchi and H.~Namkoong.
\newblock Learning models with uniform performance via distributionally robust optimization.
\newblock \emph{The Annals of Statistics}, 49\penalty0 (3):\penalty0 1378--1406, 2021.

\bibitem[Kuhn et~al.(2025)Kuhn, Shafiee, and Wiesemann]{KSW_DRO2025}
D.~Kuhn, S.~Shafiee, and W.~Wiesemann.
\newblock Distributionally robust optimization.
\newblock \emph{Acta Numerica}, 34:\penalty0 579--804, 2025.

\bibitem[Manning and Sch{\"u}tze(1999)]{MS_FSNLP1999}
C.~D. Manning and H.~Sch{\"u}tze.
\newblock \emph{Foundations of Statistical Natural Language Processing}.
\newblock The {MIT} Press, 1999.

\bibitem[Mohajerin~Esfahani and Kuhn(2018)]{MK_MM2018}
P.~Mohajerin~Esfahani and D.~Kuhn.
\newblock Data-driven distributionally robust optimization using the wasserstein metric: Performance guarantees and tractable reformulations.
\newblock \emph{Mathematical Programming}, 171\penalty0 (1):\penalty0 115--166, 2018.

\bibitem[{MOSEK ApS}(2024{\natexlab{a}})]{Mosek_11}
{MOSEK ApS}.
\newblock {MOSEK} {O}ptimizer {API} for {Python} 11.0.29.
\newblock 2024{\natexlab{a}}.
\newblock URL \url{https://docs.mosek.com/11.0/pythonapi/index.html}.

\bibitem[{MOSEK ApS}(2024{\natexlab{b}})]{Mosek_cookbook}
{MOSEK ApS}.
\newblock {MOSEK} {M}odeling {C}ookbook~3.3.0, 2024{\natexlab{b}}.
\newblock URL \url{https://docs.mosek.com/modeling-cookbook/}.

\bibitem[Rockafellar and Wets(1998)]{RW_VA1998}
R.~T. Rockafellar and R.~J. Wets.
\newblock \emph{Variational Analysis}.
\newblock Springer, 1998.

\bibitem[Sakai(2025)]{S_MSS2025}
T.~Sakai.
\newblock The probability smoothing problem: {C}haracterizations of the {L}aplace method.
\newblock \emph{Mathematical Social Sciences}, 135:\penalty0 102409, 2025.

\bibitem[Shafieezadeh~Abadeh et~al.(2015)Shafieezadeh~Abadeh, Mohajerin~Esfahani, and Kuhn]{SMPK_NIPS2015}
S.~Shafieezadeh~Abadeh, P.~M. Mohajerin~Esfahani, and D.~Kuhn.
\newblock Distributionally robust logistic regression.
\newblock \emph{Advances in {N}eural {I}nformation {P}rocessing {S}ystems}, 28, 2015.

\bibitem[Witten and Bell(2002)]{WB_IEEE2002}
I.~H. Witten and T.~C. Bell.
\newblock The zero-frequency problem: Estimating the probabilities of novel events in adaptive text compression.
\newblock \emph{{IEEE} {T}ransactions on {I}nformation {T}heory}, 37\penalty0 (4):\penalty0 1085--1094, 2002.

\end{thebibliography}

\newpage

\appendix

\section{Convex Reformulation of $q$-DRO}\label{app:reformulation}

We provide a detailed derivation of the convex reformulation of \eqref{prob:q-DRO}.
We first introduce auxiliary variable $s_j$ to represent the logarithmic term $-\log{x_j}$, and rewrite the problem as
\begin{align*}
  \underset{(\vt{x},\vt{\lambda},\beta,\vt{s})}{\mini} \quad & \sum_{j=1}^{n}\hat{p}_{j}(s_j+\lambda_{j}) + \varepsilon \cdot \| \vt{s}-\beta\vt{1}+\vt{\lambda} \|_{q}^{\bullet}\\
  \st \quad & \sum_{j=1}^{n}x_{j}=1,\\
  & -\log{x_j}\le s_j \quad (j \in N),\\
  & \lambda_{j} \ge 0 \quad (j \in N),\\
  & x_{j} \ge 0 \quad (j \in N).
\end{align*}

This problem is convex since the objective function is the sum of a linear term and a convex norm term and the constraint $-\log{x_j} \le s_j$ defines a convex set representable via exponential cone constraints.

We now show that the constraint $-\log{x_j} \le s_j$ holds with equality for all $j \in N$ at any optimal solution.
Suppose for contradiction that at an optimal solution $(\vt{x},\vt{\lambda},\beta,\vt{s})$, there exists some $j \in N$ such that $-\log{x_j} < s_j$.
Then, since $x_j\ge \exp(-s_j)$ holds with strict inequality for at least one index, we have $M=\sum_{j=1}^{n}\exp(-s_j)<\sum_{j=1}^{n}x_j=1$.
Then, we construct another solution $(\bar{\vt{x}},\bar{\vt{\lambda}},\bar{\beta},\bar{\vt{s}})$ by shifting as follows:
\begin{align*}
  \bar{s}_j=s_j-\delta, \quad -\log{\bar{x}_j}=\bar{s}_j, \quad \bar{\lambda}_j=\lambda_j, \quad \bar{\beta}=\beta-\delta,
\end{align*}
where $\delta=\log(1/M)>0$.
Since 
\begin{align*}
  \bar{x}_j=\exp(-\bar{s}_j)>0 \text{ and }
  \sum_{j=1}^{n}\bar{x}_j=\sum_{j=1}^{n}\exp(-\bar{s}_j)=\sum_{j=1}^{n}\exp(-s_j+\delta)=M\cdot \exp(\delta)=1,
\end{align*}
the solution $(\bar{\vt{x}},\bar{\vt{\lambda}},\bar{\beta},\bar{\vt{s}})$ is feasible.
The objective value at $(\bar{\vt{x}},\bar{\vt{\lambda}},\bar{\beta},\bar{\vt{s}})$ is 
\begin{align*}
  \sum_{j=1}^{n}\hat{p}_{j}(\bar{s}_j+\bar{\lambda}_{j}) + \varepsilon \cdot \| \bar{\vt{s}}-\bar{\beta}\vt{1}+\bar{\vt{\lambda}} \|_{q}^{\bullet}
  &=\sum_{j=1}^{n}\hat{p}_{j}(s_j+\lambda_{j}) -\delta + \varepsilon \cdot \| \vt{s}-\beta\vt{1}+\vt{\lambda} \|_{q}^{\bullet}\\
  &<\sum_{j=1}^{n}\hat{p}_{j}(s_j+\lambda_{j}) + \varepsilon \cdot \| \vt{s}-\beta\vt{1}+\vt{\lambda} \|_{q}^{\bullet},
\end{align*}
which contradicts the optimality of $(\vt{x},\vt{\lambda},\beta,\vt{s})$.
Thus, at an optimal solution, we have $-\log{x_j}=s_j$ for all $j \in N$.

\newpage
\section{Monotonicity of Subgradients}\label{app:monotonicity}

We prove the monotonicity of the subgradients $\vt{y}^{(q^{*})} \in \partial \|\vt{u}\|_{q^{*}}$ with respect to the components $u_j=-\log{x_j}-\beta$ under Assumption~\ref{ass:don-degen}.

For $q^{*}=1$, the subgradient $\vt{y}^{(1)}$ is given by
\[y_j^{(1)}\in
\begin{cases}
\{\mathrm{sgn}(u_j)\}, & u_j\neq0,\\
[-1,1], & u_j=0.
\end{cases}\]
Assume $u_i < u_j$.
We examine all possible cases for $u_i$ and $u_j$.
\begin{itemize}
  \item When $u_i < 0 < u_j$, we have $y_i^{(1)} = -1 < 1 = y_j^{(1)}$.
  \item When $u_i < u_j < 0$, we have $y_i^{(1)} = -1 = y_j^{(1)}$.
  \item When $0 < u_i < u_j$, we have $y_i^{(1)} = 1 = y_j^{(1)}$.
  \item When $u_i = 0 < u_j$, we have $y_i^{(1)} =c \le 1 = y_j^{(1)}$ for some $c \in [-1,1]$.
  \item When $u_i < 0 = u_j$, we have $y_i^{(1)} = -1 \le c=y_j^{(1)} $ for some $c \in [-1,1]$.
\end{itemize}
Thus, in all cases, we have $y_i^{(1)} \le y_j^{(1)}$ when $u_i < u_j$.

For $q^{*} \in (1, \infty)$, the subgradient coincides with the gradient:
\[y_j^{(q^{*})}=\dfrac{|u_j|^{q^{*}-1}\mathrm{sgn}(u_j)}{\|\vt{u}\|_{q^{*}}^{q^{*}-1}}.\]
Assume $u_i \le u_j$.
\begin{itemize}
  \item When $u_i < u_j$, we have $|u_i|^{q^{*}-1}\mathrm{sgn}(u_i) < |u_j|^{q^{*}-1}\mathrm{sgn}(u_j)$, which implies $y_i^{(q^{*})} < y_j^{(q^{*})}$.
  \item When $u_i = u_j$, we have $y_i^{(q^{*})} = y_j^{(q^{*})}$.
\end{itemize}
Thus, in all cases, we have $y_i^{(q^{*})} \le y_j^{(q^{*})}$ when $u_i \le u_j$.

For $q^{*} = \infty$, let $I(\vt{u})=\{j\in N:\ |u_j|=\|\vt{u}\|_{\infty}\}$, the subgradient $\vt{y}^{(\infty)}$ is given by
\[y_j^{(\infty)}=
\begin{cases}
c_j\,\mathrm{sgn}(u_j), & j\in I(\vt{u}),\\
0, & j\notin I(\vt{u}),
\end{cases}\]
where there exists coefficient $c_j\ge0$ such that $\sum_{j\in I(\vt{u})}c_j=1$.
Assume $u_i < u_j$.
\begin{itemize}
  \item When $i,j \in I(\vt{u})$, the only possibility is $u_i=-\|\vt{u}\|_{\infty}$ and $u_j=\|\vt{u}\|_{\infty}$.
    Thus, we have $y_i^{(\infty)}=c_i\,\mathrm{sgn}(u_i)=-c_i$ and $y_j^{(\infty)}=c_j\,\mathrm{sgn}(u_j)=c_j$ for some $c_i,c_j \ge 0$.
    This implies $y_i^{(\infty)}=-c_{i} \le c_{j}=y_j^{(\infty)}$.
  \item When $i \in I(\vt{u})$ and $j \notin I(\vt{u})$, it follows that $y_i^{(\infty)}=c_i\,\mathrm{sgn}(u_i)$ for some $c_i \ge 0$ and $y_j^{(\infty)}=0$.
    Since $u_i < u_j$, we have $\mathrm{sgn}(u_i) =-1$, which implies $y_i^{(\infty)}=-c_{i} \le 0 = y_j^{(\infty)}$. 
  \item When $i \notin I(\vt{u})$ and $j \in I(\vt{u})$, it follows that $y_i^{(\infty)}=0$ and $y_j^{(\infty)}=c_j\,\mathrm{sgn}(u_j)$ for some $c_j \ge 0$.
    Since $u_i < u_j$, we have $\mathrm{sgn}(u_j) = 1$, which implies $y_i^{(\infty)}=0 \le c_j=y_j^{(\infty)}$.
  \item When $i,j \notin I(\vt{u})$, we have $y_i^{(\infty)}=0$ and $y_j^{(\infty)}=0$, which implies $y_i^{(\infty)} = y_j^{(\infty)}$.
\end{itemize}
Thus, in all cases, we have $y_i^{(\infty)} \le y_j^{(\infty)}$ when $u_i < u_j$.
\end{document}